\documentclass[11pt,twoside,reqno]{amsart}
\usepackage{epsfig}
\usepackage{hyperref}
\usepackage[capitalize]{cleveref}
\usepackage{amssymb}
\usepackage{graphicx}
\usepackage{enumerate}
\usepackage{multirow}
\usepackage{amsmath,color}
\usepackage{hyperref}
\usepackage{url}
\usepackage{mleftright}
\usepackage{enumitem}
\usepackage{amsmath,color,tikz,cancel}
\calclayout

\newtheorem{theorem}{\bf Theorem}
\newtheorem{lemma}[theorem]{\bf Lemma}
\newtheorem{proposition}[theorem]{\bf Proposition}
\newtheorem{cor}[theorem]{\bf Corollary}

\newtheorem{ex}{\bf Example}

\newtheorem{rem}{\bf Remark}

\usepackage{float}
\usepackage{calc,tikz,amsmath}
\usetikzlibrary{arrows,decorations.markings,positioning,shapes.geometric,}
\tikzstyle{vertex}=[circle, draw, inner sep=1pt, minimum size=4pt]

\tikzstyle{ann} = [fill=white,font=\footnotesize,inner sep=1pt]
\tikzstyle{arrow} = [thick,<-->,>=stealth]

\begin{document}
\vspace{1.3cm}

\title
{On the Largest intersecting set in $GL_2(q)$ and  some of its subgroups}

\author{Milad Ahanjideh}
\address{ Department of Industrial Engineering, Bo\u{g}azi\c{c}i University, Turkey}
\email{ahanjidm@gmail.com}
\keywords{ Erd\H{o}s-Ko-Rado property, the general linear group, intersecting set}

\begin{abstract}
Let $q$ be a power of a prime number and $V$ be the $2$-dimensional column vector space over a finite field $\mathbb{F}_{q}$.
Assume that $SL_2(V)<G\leq GL_2(V)$. In this paper we prove an Erd{\H{o}}s-Ko-Rado theorem for intersecting sets of G and we show that
every maximum intersecting set of $G$
 is either a coset of the stabilizer of a point or a coset of $\mathcal{G}_{\langle w\rangle}$, where $\mathcal{G}_{\langle w\rangle}=\{M\in G:\forall v\in V, Mv-v\in \langle w\rangle\}$, for some $w\in V\setminus \{0\}$. It is also shown that every intersecting set of $G$    is  contained in  a maximum intersecting set.
 \end{abstract}
\subjclass[2010]{05D05, 20B35}
\maketitle


\pagestyle{myheadings} \markboth{\centerline {\scriptsize On the Largest intersecting set in $GL_2(q)$  and  some of its subgroups }}
         {\centerline {\scriptsize On the Largest intersecting set in $GL_2(q)$  and  some of its subgroups }}


\bigskip
\bigskip
\medskip

\section{Introduction}
\vspace{.5cm}
One of the fundamental results in extremal set theory is the Erd{\H{o}}s-Ko-Rado theorem which was proved in 1961~\cite{EKR}. This theorem is concerned with obtaining the size of the largest family of subsets of size $r$ from $\{1,\ldots, n\}$ such that any two of these subsets have a nonempty intersection. The famous Erd{\H{o}}s-Ko-Rado theorem states that if $n\geq 2r$, then the largest family has size ${n-1 \choose r-1}$, and for $n>2r$, the only families that have this size are the families of all subsets containing a fixed element from $\{1,\ldots, n\}$.

There are many generalizations of this theorem for other mathematical objects. In this paper, we consider the Erd{\H{o}}s-Ko-Rado theorem for the general linear group of dimension $2$ and its subgroups containing the special linear group.

 Let $\Omega$ be a finite set and $K$ be a permutation group on it. We say that a subset $A$ of $K$ is {\it intersecting} if for any $\delta,\tau \in A$, there exists $x\in \Omega$ such that $\delta(x)=\tau(x)$.  It is obvious that the stabilizer of a point is an intersecting set. If the size of every intersecting set of a group $G$ is at most the size of the largest stabilizer of a point, we say $G$ has the {\it EKR property} while $G$ has the {\it strict EKR property} if  every maximum  intersecting set is a coset of a stabilizer of a point.

 A natural question that arises here is to determine which groups have the (strict) EKR property. Our first theorem is in line with answering this question. Before describing our theorems, let us mention a few interesting results.

 The first result in this area that dates back to 1977 is due to Frankl and Deza ~\cite{Frankl}. Let $S_n$ denote the symmetric group of degree $n$. Frankl and Deza  proved that $S_n$ has the EKR property. After that, Cameron and Ku \cite{cameron} showed that  $S_n$ has the strict EKR property as well which answered affirmatively to the conjecture posed in~\cite{Frankl}. This result was also proved independently by Larose and Malvenuto~\cite{LM}.

 It has been  proved that the alternating groups have the strict EKR property~\cite{Ku}. In ~\cite{Ahanjideh}, the authors showed that $SL_2(q)$ acting on a $2$-dimensional vector space over $\mathbb{F}_{q}$ has the strict EKR property. From combined results of \cite{2Ahmadi,MeagherS11}, it is known that the size of the  maximum intersecting set in $PSL_2(q)$ is $q(q-1)/2$. As a step forward, Ling  et al.~\cite{sin} presented the characterization of the maximum intersecting sets in  $PSL_2(q)$ whenever $q>3$ is an odd prime power.

 Furthermore, it has been shown that the projective general linear groups $PGL_2(q)$ have the strict EKR property~\cite{MeagherS11}, while the groups $PGL_3(q)$ do not so~\cite{MeagherS13}.  Finally, Spiga showed that the maximum intersecting sets of $PGL_{n+1}(q)$ acting on the points of the projective space $\mathbb{P}^n_q$  are the cosets of the
stabilizer of a point and the cosets of the stabilizer of a hyperplan \cite{Spiga}. For more studies, we refer the reader to~\cite{Godsil,Meagherps,Meaghertip}.

 It is worth mentioning that the group action plays an essential role to determine which groups  have (strict) EKR property. Because a group can have the EKR property under an action while it does not have the same property under the other action. We refer the reader to see~\cite{Ahmadi} for finding an example of such  groups.

In~\cite{Ahanjideh}, the authors showed that the EKR property holds for $GL_2(V)$, however their result is incomplete. In this paper, we inspire  their method and employ the properties of the linear groups and group-action for generalizing their result and correcting the flaw.  Let $SL_2(V)<G\leq GL_2(V)$. We show that $G$ has the EKR property and does not have the strict EKR property. Furthermore, we prove that every intersecting set $\mathcal{F}$ of $G$ is  contained in an intersecting set of the maximal size.

Note that the authors of ~\cite{Ahanjideh} showed that $GL_n(q)$ of arbitrary dimension admits EKR property. We do not already have any significant evidence whether $GL_n(q)$ has the strict EKR property or not.

\section{Additional notation and auxiliary Lemmas}
\vspace{.5cm}
 In this section, we recall some notation and definitions.  Also, we prove several lemmas in order to prove our main  results.

Assume that $p$ is  a prime number, $k$ is a positive integer and $q=p^k$. We denote a finite field of order $q$ by $\mathbb{F}_{q}$.
Let $V=\mathbb{F}_{q}^{2}$ and $G$ be a finite group such that $SL_2(V)<G\leq GL_2(V)$. Throughout this paper, assume that $\mathcal{F}$ is an intersecting set of $G$ acting naturally on $\mathbb{F}_{q}^{2}$ and $\mathcal{F}_0=\mathcal{F}\setminus\{Id\}$.
For a basis $B$ of $V$ and $g\in G$, $[g]_B$ stands for the matrix representation corresponding to $g$ and for $w\in V\setminus \{0\}$, ${\langle w\rangle}$ stands for the $1$-dimensional subspace of $V$ with $\{w\}$ as its basis. We denote the stabilizer subgroup of a point $v\in V$ by $G_v$ while $G_{\langle w\rangle}$ denotes the stabilizer subgroup of a subspace ${\langle w\rangle}$.

 For every $w\in V\setminus \{0\}$, set $$\mathcal{G}_{\langle w\rangle}=\{M\in G:\forall v\in V, Mv-v\in \langle w\rangle\}.$$
We define ${\rm det}(G)$ as the set of determinants of  all elements of $G$.
\begin{rem}\label{rem}
     Let $\mathcal{F}$ be an intersecting set of $G$. In this paper we assume that $Id\in \mathcal{F}$. Because if $x \in \mathcal{F}$, then $Id=x^{-1}x \in x^{-1}\mathcal{F}$ and $|\mathcal{F}|=|x^{-1}\mathcal{F}|$. So $x^{-1}\mathcal{F}$ is an intersecting set containing $Id$ with the same size as $|\mathcal{F}|$. Now we can use $x^{-1}\mathcal{F}$ instead of $\mathcal{F}$.
\end{rem}
\begin{lemma}\label{repOfH}
    For $w\in V\setminus \{0\}$, $\mathcal{G}_{\langle w\rangle}$ is  a subgroup of $G$ and $|\mathcal{G}_{\langle w\rangle}|=qd$, where $d= |{\rm det}(G)|$.
\end{lemma}
\begin{proof}
    Since $Id\in \mathcal{G}_{\langle w\rangle}$, $\mathcal{G}_{\langle w\rangle}\neq \emptyset$. Let $M_1,M_2\in \mathcal{G}_{\langle w\rangle}$. Then, for every $v\in V$, $M_1v-v,M_2v-v\in \langle w\rangle$. Thus, $M_1(M_2v)-M_2v\in \langle w\rangle$, and hence $M_1M_2v-M_2v+M_2v-v\in \langle w\rangle$. This shows that $M_1M_2v-v\in \langle w\rangle$. Therefore $M_1M_2\in \mathcal{G}_{\langle w\rangle}$.\\
    Also, $M^{-1}_1v\in V$, so $M_1M^{-1}_1v-M^{-1}_1v\in \langle w\rangle$ and consequently, $M^{-1}_1v-v\in \langle w\rangle$. This implies that $M^{-1}_1\in \mathcal{G}_{\langle w\rangle}$ and hence $\mathcal{G}_{\langle w\rangle}\leq G$.\\
    Now we are going to prove the second part. Assume that $B=\{v,w\}$ is a  basis of $V$. Let $g\in \mathcal{G}_{\langle w\rangle}\setminus \{Id\}$. Since $gv-v,gw-w$ belong to ${\langle w\rangle}$, we have
    $$[g]_B=\begin{bmatrix}
1 & 0 \\
c & \mu
\end{bmatrix},$$
where $\mu \in {\rm det}(G)$ and $c\in \mathbb{F}_q$. For every $\begin{pmatrix}
\alpha \\
\beta
\end{pmatrix} \in V$, we can see that
$$\begin{bmatrix}
1 & 0 \\
c & \mu
\end{bmatrix}
\begin{pmatrix}
\alpha \\
\beta
\end{pmatrix}-
\begin{pmatrix}
\alpha \\
\beta
\end{pmatrix}\in \langle w\rangle .
$$
 One can check at once that for every $\mu \in {\rm det}(G)$ and $ c\in \mathbb{F}_q$, if $g \in GL_2(V)$ such that
 $[g]_B=\begin{bmatrix}
1 & 0 \\
c & \mu
\end{bmatrix}$, then $g\in G$. This guarantees that $$\mathcal{G}_{\langle w\rangle}=\{g\in G:[g]_B=\begin{bmatrix}
1 & 0 \\
c & \mu
\end{bmatrix}, {\rm where}~ \mu\in {\rm det}(G)\ {\rm and}~c\in \mathbb{F}_q\},$$ which gives $|\mathcal{G}_{\langle w\rangle}|=q|{\rm det}(G)|$, as wanted.
\end{proof}
\begin{lemma}\label{stab}
    Let $v\in V\setminus \{0\}$. Then $|G_v|=q|{\rm det}(G)|$.
\end{lemma}
\begin{proof}
    Since $G$ acts transitively on $V\setminus \{0\}$, we have $[G:G_v]=q^2-1$. Observe that $|G|=q(q^2-1)|{\rm det}(G)|$. So, the lemma follows.

\end{proof}
\begin{lemma}
    Let $w\in V\setminus \{0\}$. Then $\mathcal{G}_{\langle w\rangle}$ is an intersecting set of $G$.
\end{lemma}
\begin{proof}
    Let $B=\{v,w\}$ be a basis of $V$. As mentioned in the proof of Lemma \ref{repOfH}, we have
    $\mathcal{G}_{\langle w\rangle}=\{g\in G:[g]_B=\begin{bmatrix}
1 & 0 \\
c & \mu
\end{bmatrix}, {\rm where}~ \mu\in {\rm det}(G)\ {\rm and}~c\in \mathbb{F}_q\}$. Let $g_1,g_2\in \mathcal{G}_{\langle w\rangle}$. Then there exist $\mu_1,\mu_2\in {\rm det}(G)$ and $c_1,c_2\in \mathbb{F}_q$ such that
$[g_1]_B=\begin{bmatrix}
1 & 0 \\
c_1 & \mu_1
\end{bmatrix}$ and
$[g_2]_B=\begin{bmatrix}
1 & 0 \\
c_2 & \mu_2
\end{bmatrix}$. If $\mu_1=\mu_2$, then $g_1w=g_2w$. Otherwise $g_1(v+\alpha w)=g_2(v+\alpha w)$, where $\alpha=(\mu_2-\mu_1)^{-1}(c_1-c_2)$. Thus $\mathcal{G}_{\langle w\rangle}$ is an intersecting set, as wanted.

    \end{proof}
\begin{lemma}\label{new}
    Let $B=\{v,w\}$ be a  basis of $V$ and $\mathcal{F}$ be an intersecting set of $G$. If $ x\in \mathcal{F}_0 \cap G_v$ and $y \in \mathcal{F}_0 \cap G_w$, then
    $[x]_B=\begin{bmatrix}
1 & c \\
0 & \lambda
\end{bmatrix}$
and
$[y]_B=\begin{bmatrix}
1+c\alpha & 0 \\
(\lambda-1)\alpha & 1
\end{bmatrix}$
for some $\lambda\in {\rm det}(G)$,  $c\in \mathbb{F}_q$ and $\alpha \in \mathbb{F}^{\ast}_q$.
\end{lemma}
\begin{proof}
    Since $x\in G_v\setminus\{Id\}$ and $y\in G_w\setminus G_v$, we have
    $[x]_B=\begin{bmatrix}
1 & c \\
0 & \lambda
\end{bmatrix}$
and
$[y]_B=\begin{bmatrix}
\lambda' & 0 \\
c' & 1
\end{bmatrix}$,
where $\lambda,\lambda'\in {\rm det}(G)$ and $c,c'\in \mathbb{F}_q$. The set $\mathcal{F}$ is an intersecting set of $G$ and $x,y\in \mathcal{F}$. So, there exists $u\in V\setminus \{0\}$ such that $x(u)=y(u)$. Note that $V={\langle v,w \rangle}$. Hence $u=av+bw$ for some $a,b\in \mathbb{F}^{\ast}_q$ and $ax(v+a^{-1}bw)=x(u)=y(u)=ay(v+a^{-1}bw)$. Therefore $x(v+\alpha w)=y(v+\alpha w)$, where $\alpha=a^{-1}b$. It follows that $(1+c\alpha)v+\lambda\alpha w=\lambda'v+(c'+\alpha)w$. Consequently, $\lambda'=1+c\alpha$ and $c'=(\lambda-1)\alpha$. Also this shows that
\begin{equation}\label{51}
    \mathcal{F}_0\cap G_w\subseteq \{ z\in G~:~ [z]_B=\begin{bmatrix}
1+c\gamma & 0 \\
(\lambda-1)\gamma & 1
\end{bmatrix},~
{\rm where}~ \gamma\in \mathbb{F}_q^{\ast} \}.
\end{equation}
\end{proof}
\begin{cor}\label{ned}
    Let $B=\{v,w\}$ be a  basis of $V$ and $\mathcal{F}$ be an intersecting set of $G$. If $\mathcal{F} \cap G_w \neq \{Id\}$ and  $ x\in \mathcal{F}_0 \cap G_v$ such that
    $[x]_B=\begin{bmatrix}
1 & 0 \\
0 & \lambda
\end{bmatrix}$,
for some $\lambda\in {\rm det}(G)$,   then $\mathcal{F} \cap G_v \subseteq G_{\langle w\rangle}$.
\end{cor}
\begin{proof} Since $\mathcal{F}\cap G_w \neq \{Id\}$, there exists an element $ y \in \mathcal{F}_0\cap G_w$. By Lemma \ref{new}, \begin{equation} \label{501}[y]_B=\begin{bmatrix}
1 & 0 \\
(\lambda-1)\alpha & 1
\end{bmatrix},\end{equation}
for some  $\alpha \in \mathbb{F}^{\ast}_q$. Now, let $z \in \mathcal{F}_0 \cap G_v$. Then, $ [z]_B=\begin{bmatrix}
1 & c' \\
0 & \lambda'
\end{bmatrix}$,
for some $\lambda'\in {\rm det}(G)$ and  $c'\in \mathbb{F}_q$. By \eqref{51}, we know that
\begin{equation}\nonumber
    \mathcal{F}_0\cap G_w\subseteq \{ w\in G~:~ [w]_B=\begin{bmatrix}
1+c'\gamma & 0 \\
(\lambda'-1)\gamma & 1
\end{bmatrix},~
{\rm where}~ \gamma\in \mathbb{F}_q^{\ast} \}.
\end{equation}
 Hence, $[y]_B=\begin{bmatrix}
1+c'\mu & 0 \\
(\lambda'-1)\mu & 1
\end{bmatrix},$
for some  $\mu \in \mathbb{F}^{\ast}_q$. So, \eqref{501} forces $1=1+c'\mu$. Since $ \mu \neq 0$, $c'=0$. Thus, $z \in G_{\langle w\rangle}$. This shows that $\mathcal{F} \cap G_v \subseteq G_{\langle w\rangle}$, as wanted.
\end{proof}
\begin{cor}\label{cor51}
    Let $B=\{v,w\}$ be a basis of $V$ and $\mathcal{F}$ be an intersecting set of $G$. If $\mathcal{F}_0\cap G_{v}\neq \emptyset$, then $|\mathcal{F}_0 \cap G_w |\leq\epsilon$, where
$$ \epsilon=
\begin{cases}
    q-1, & {\rm if}~ \mathcal{F}_0 \cap G_v\subseteq G_{\langle w\rangle}, \\
    d-1,  & {\rm otherwise},
\end{cases}
$$
and $d= |{\rm det}(G)|$. Also, if $\mathcal{F}_0\cap G_{w} \not\subseteq  G_{\langle v\rangle} $, then $|\mathcal{F}_0\cap G_v|\leq d-1$.
\end{cor}
\begin{proof}
Let $x\in \mathcal{F}_0\cap G_{v}$. By Lemma \ref{new}, $[x]_B=\begin{bmatrix}
1 & c \\
0 & \lambda
\end{bmatrix}$ where $\lambda\in {\rm det}(G)$ and $c\in \mathbb{F}_q$. By (\ref{51}), we have

$$\mathcal{F}_0\cap G_w\subseteq \{ z\in G~:~ [z]_B=\begin{bmatrix}
1+c\gamma & 0 \\
(\lambda-1)\gamma & 1
\end{bmatrix},~
{\rm where}~ \gamma\in \mathbb{F}_q^{\ast} \}.$$
Thus $|\mathcal{F}_0 \cap G_w|\leq q-1$. If $\mathcal{F} \cap G_v\not\subseteq G_{\langle w\rangle}$, then $c\neq0$.
Also, $1+c\gamma\in {\rm det}(G)$. So $\gamma\in c^{-1}(1-{\rm det}(G))$. Since $\gamma\neq 0$, $\gamma\in c^{-1}(1-{\rm det}(G))\setminus \{0\}$. It follows from (\ref{51}) that $|\mathcal{F}_0 \cap G_w |\leq d-1$, as desired. The same argument completes the proof for $|\mathcal{F}_0\cap G_v|$.
\end{proof}
\begin{ex} Let $V=\mathbb{F}_{3}^{2}$ and  $G=GL_2(V)$. Consider $v=(1,0)$, $w=(0,1)$ and $w_1=(1,-1)$ as the elements of $V$.  Then $B=\{v,w\}$ is a basis of $V$. Set
\begin{equation*}
\mathcal{F}_1=\left\{\left[\begin{array}{ll}
1 & 0 \\
0 & 1
\end{array}\right],\left[\begin{array}{cc}
1 & -1 \\
0 & -1
\end{array}\right],\left[\begin{array}{cc}
-1 & 0 \\
-1 & 1
\end{array}\right],\left[\begin{array}{ll}
0 & 1 \\
1 & 0
\end{array}\right],\left[\begin{array}{cc}
-1 & 1 \\
-1 & 0
\end{array}\right], \left[\begin{array}{cc}
0 & -1 \\
1 & -1
\end{array}\right]
\right\}
\end{equation*} and
$$\mathcal{F}=\{M\in GL_2(V): ~[M]_B\in \mathcal{F}_1\}.$$
One can check at once that $\mathcal{F}=\mathcal{G}_{\langle w_1\rangle}$ and $d=|{\rm det}(G)|=2$. Also,  $\mathcal{F}_0 \cap G_v \not \subseteq G_{\langle w\rangle}$ and $\mathcal{F}_0 \cap G_{ w}=\left\{\left[\begin{array}{cc}
-1 & 0 \\
-1 & 1
\end{array}\right]\right\}$. Hence, $|\mathcal{F}_0 \cap G_{ w}|=1$.
\end{ex}
\begin{cor}\label{cor}
Let $B=\{v,w\}$ be a basis of $V$ and $\mathcal{F}$ be an intersecting set of $G$. If $x,y\in \mathcal{F}_0$ such that
$ x\in G_v\cap G_{\langle w\rangle}$ and $ y\in G_w$, then $y(v)=v+c'w$, for some $c'\in \mathbb{F}_q$.
\end{cor}
\begin{proof}
The proof follows immediately from (\ref{51}).
\end{proof}
\section{Erd{\H{o}}s-Ko-Rado Theorem for the general linear group and its subgroups}
\vspace{.5cm} The purpose of this section is to prove analogues of
the Erd{\H{o}}s-Ko-Rado theorem for the general linear group of
dimension $2$ and some of  its subgroups. The following observations  will be needed throughout this section.
\begin{proposition} \label{extra1} Let $B=\{v,w\}$ be a basis of $V$ and $\mathcal{F}$ be an intersecting set of $G$ such that $\mathcal{F}_0 \cap G_v \neq \emptyset $ and $
\mathcal{F} \subseteq
G_{\langle w \rangle}$. Then, $\mathcal{F}\subseteq
\mathcal{G}_{\langle w \rangle}$.
\end{proposition}
\begin{proof}By Remark~\ref{rem}, $Id\in \mathcal{F}$. Set $w_t=v+tw$,  where $t\in \mathbb{F}_q$. Since $\mathcal{F} \subseteq
G_{\langle w \rangle}$, we get   $\mathcal{F} \cap  G_{w_t}\subseteq G_{\langle w \rangle}$ for  every $t\in \mathbb{F}_q$.
Note that $\{w_t:t\in \mathcal{F}_q\}\cup \{w\}$ are pairwise
linearly independent, and for every $u\in V$, either $u\in {\langle
w \rangle}$ or there exists $t\in \mathbb{F}_q$ such that $u\in
{\langle {w_t} \rangle}$. Since $\mathcal{F}$ is an intersecting
set of $G$, for every $g\in \mathcal{F}$, there exists an element
$u\in V\setminus \{0\}$ such that $g(u)=Id(u)=u$. So, $g\in G_u$.
Consequently,
        \begin{equation}\label{union}
        \mathcal{F}=(\dot{\bigcup_{t\in \mathbb{F}_q}} \mathcal{F}_0 \cap G_{w_t})\dot{\cup}(\mathcal{F}_0 \cap G_{w})\dot{\cup} \{Id\}.
        \end{equation}
If $g\in \mathcal{F}_0\cap G_v$, then $gv-v=0\in {\langle w
\rangle}$ and $g(\alpha v+w)-(\alpha v+w)\in {\langle w \rangle}$,
for every $\alpha\in\mathbb{F}_q$, because $g\in \mathcal{F}_0\cap
G_v\subseteq G_{\langle w \rangle}$. So $g\in \mathcal{G}_{\langle
w \rangle}$. If $t\in\mathbb{F}^{\ast}_q$ and
$g\in\mathcal{F}_0\cap G_{w_t}$, then
\begin{equation*}
    \begin{aligned}
        gv-v =g(v+tw-tw)-v=v+tw-tg(w)-v & \\=tw-tg(w)\in {\langle w \rangle},
    \end{aligned}
\end{equation*}
because $g\in \mathcal{F}_0 \cap  G_{w_t}\subseteq G_{\langle w
\rangle}$.  Also
\begin{equation*}
    \begin{aligned}
        g(\alpha v +w)-(\alpha v+w) & = \alpha(v+tw)+(-t\alpha+1)g(w)-(\alpha v+w) \\  & =(\alpha t-1)w+(-t\alpha+1)g(w)\in {\langle w \rangle},
    \end{aligned}
\end{equation*}
for every $\alpha \in \mathbb{F}_q$. Therefore $g\in
\mathcal{G}_{\langle w \rangle}$. Observe that if $g\in G_w$, then
$g\in \mathcal{G}_{\langle w \rangle}$. Thus $\mathcal{F}\subseteq
\mathcal{G}_{\langle w \rangle}$, as wanted.
\end{proof}
\begin{proposition} \label{extra} Let $B=\{v,w\}$ be a basis of $V$ and $\mathcal{F}$ be an intersecting set of $G$. If $\mathcal{F}_0\cap G_w\neq\emptyset $ and  $\emptyset \neq \mathcal{F}_0\cap G_v \subseteq G_{\langle w \rangle}$, then $\mathcal{F}\subseteq \mathcal{G}_{\langle w \rangle}$.
\end{proposition}
\begin{proof}   Let $x\in \mathcal{F}_0 \cap  G_v$ and $y \in \mathcal{F}_0\cap G_w$.    Then, $y\not\in G_v$ and
    $[y]_B=\begin{bmatrix}
\lambda & 0 \\
c' & 1
\end{bmatrix}$, where $c'\in \mathbb{F}_q$ and $\lambda\in\mathbb{F}^{\ast}_q$.
Since $\mathcal{F}\cap G_v \subseteq G_{\langle w \rangle}$, $\lambda=1$ by Corollary \ref{cor}.
 For $t\in \mathbb{F}_q$, let $w_t=v+tw$.
        \begin{itemize}[leftmargin=*]
        \item[(i)] If there exists  $t\in \mathbb{F}^{\ast}_q$ such that $\mathcal{F}\cap G_{w_t} \not\subseteq G_{\langle w \rangle}$, then there is  $z\in (\mathcal{F}\cap G_{w_t})\setminus G_{\langle w \rangle}$. Since $z,y\in \mathcal{F}_0$ and $\mathcal{F}$ is an intersecting set of $G$, there exists $u\in V\setminus \{0\}$ such that $y(u)=z(u)$. If $u=v$, then $z(v)=y(v)=v+c'w$. We can see that $v+tw=z(v+tw)=z(v)+tz(w)=v+c'w+tz(w)$. Therefore $z(w)= t^{-1}(tw-c'w)\in {\langle w \rangle}$. So $z\in G_{\langle w \rangle}$. This is a contradiction. If $u=v+\alpha w$, for some $\alpha\in \mathbb{F}^{\ast}_q$, then $v+tw+(\alpha-t)z(w)=z(v+\alpha w)=y(v+\alpha w)=v+c' w+\alpha w$, for some $c'\in \mathbb{F}_q$. Thus, if $\alpha\neq t$, then $z(w)=(\alpha-t)^{-1}(c'+\alpha-t)w\in {\langle w \rangle}$, which is a contradiction. Also, if $\alpha=t$, then $y(v+tw)=z(v+tw)=v+tw$. So, $y(v)=v$. Consequently $y\in G_v$, which is a contradiction.
        \item[(ii)] Assume that for every $t\in \mathbb{F}^{\ast}_q$, $\mathcal{F} \cap  G_{w_t}\subseteq G_{\langle w
        \rangle}$. Then, \eqref{union} shows that $\mathcal{F} \subseteq
G_{\langle w \rangle}$. It follows from Proposition \ref{extra1}
that  $\mathcal{F}\subseteq \mathcal{G}_{\langle w \rangle}$, as wanted.
\end{itemize}
\end{proof}
\begin{theorem}\label{erk}
    Let $\mathcal{F}$ be an intersecting set of $G$. Then $|\mathcal{F}|\leq dq$, where $d=|{\rm det}(G)|$. In particular, $|\mathcal{F}|=dq$ if and only if $\mathcal{F}$ is either a coset of $G_v$ or a coset of $\mathcal{G}_{\langle w\rangle}$ for some $v, w\in V\setminus \{0\}$.
\end{theorem}
\begin{proof}
    According to Remark~\ref{rem}, assume that $Id\in \mathcal{F}$.  Let $x\in \mathcal{F}_0$. Since $Id\in \mathcal{F}$ and $\mathcal{F}$ is an intersecting set of $G$, there exists $v \in V\setminus \{0\}$ such that $x(v)=v$. This yields $x\in G_v$. If
    $\mathcal{F}\subseteq   G_v$, then Lemma \ref{stab} completes the proof. Now let $\mathcal{F}\not\subseteq G_v$. Note that for every $a\in\mathbb{F}^{\ast}_q$, $G_v=G_{av}$. Let $y\in \mathcal{F}\setminus G_v$. Since $Id\in \mathcal{F}$ and $y\not\in G_v$, there exists $w \in V\setminus \langle v \rangle$ such that $y(w)=w$ and $B=\{v,w\}$ is a basis of $V$. Observe that $y\in G_w\setminus \{Id\}$, so
    $[y]_B=\begin{bmatrix}
\lambda & 0 \\
c' & 1
\end{bmatrix}$, where $c'\in \mathbb{F}_q$ and $\lambda\in\mathbb{F}^{\ast}_q$.   For $t\in \mathbb{F}_q$, let $w_t=v+tw$. If there exists the  linearly independent vectors  $u,u' \in \{v,w,w_t:t \in \mathbb{F}_q^{\ast}\}$ such that $ \mathcal{F}_0\cap G_{u'}\neq \emptyset $ and  $\emptyset \neq \mathcal{F}_0\cap G_u\subseteq G_{\langle u' \rangle}$,  then Proposition \ref{extra} shows that $\mathcal{F}\subseteq \mathcal{G}_{\langle u' \rangle}$, as desired. This shows that if $|\mathcal{F}|=dq$,  then  $\mathcal{F}=\mathcal{G}_{\langle u'\rangle}$. Next assume that for every  linearly independent vectors $u,u' \in \{v,w,w_t:t \in \mathbb{F}_q^{\ast}\}$, if $\mathcal{F}_0\cap G_u \neq \emptyset$ and $\mathcal{F}_0\cap G_{u'}\neq \emptyset$, then $\mathcal{F}\cap G_u\not\subseteq G_{\langle u' \rangle}$. Thus, Corollary \ref{cor51} and (\ref{union}) show that $|\mathcal{F}|\leq (q+1)(d-1)+1<qd$. So, the theorem follows.
\end{proof}
\begin{theorem}
Every intersecting set of $G$   is  contained in an intersecting set of size $qd$.
\end{theorem}
\begin{proof} Let $\mathcal{F}$ be an intersecting set of $G$ such that $\mathcal{F}\not\subseteq  G_u$, for every $u \in V\setminus \{0\}$.
    By Remark \ref{rem}, we can assume that $Id\in \mathcal{F}$.  Let $x_0\in \mathcal{F}_0$. Since $\mathcal{F}$ is an intersecting set of $G$ and $Id\in \mathcal{F}$,  $x_0(v)=Id(v)$ for some $v\in V\setminus \{0\}$. So, we have  $x_0\in \mathcal{F}_0 \cap G_v$. Note that $\mathcal{F}\not\subseteq  G_v$. This implies that there exists  $y_0\in \mathcal{F}_0\setminus G_ v$. Now,  $y_0\in \mathcal{F}$  and $\mathcal{F}$ is an intersecting set of $G$. Hence, we can find  $w\in V\setminus {\langle v \rangle}$ such that $y_0\in G_w$. Set $w_t=v+tw$, for $t \in \mathbb{F}_q^{\ast}$ and let $B_0=\{v,w\}$.  It is clear that   $B_0$ is a  basis of $V$.
    By Lemma \ref{new},
    \begin{equation}\label{titi} [x_0]_{B_0}=\begin{bmatrix}
1 & c_0 \\
0 & \lambda_0
\end{bmatrix},\end{equation}
where $c_0\in \mathbb{F}_{q}$ and $\lambda_0\in {\rm det}(G)$. If $c_0=0$, then $\mathcal{F}\cap G_{v} \subseteq G_{\langle w\rangle}$, by Corollary \ref{ned}.  It follows from  Proposition   \ref{extra} that $\mathcal{F} \subseteq \mathcal{G}_{\langle w\rangle}$, as desired. If $\lambda_0=1$, then \eqref{51} shows that $\mathcal{F}\cap G_{w} \subseteq G_{\langle v\rangle}$. So, Proposition \ref{extra} guarantees that $\mathcal{F} \subseteq \mathcal{G}_{\langle v\rangle}$, as desired.  Finally let $c_0 \neq 0$ and $\lambda_0\neq 1$.  Then,  $x_0(v+(\lambda_0-1)c_0^{-1}w)=\lambda_0(v+(\lambda_0-1)c_0^{-1}w)$.  Thus, $  x_0 \in G_{\langle w_{\alpha}\rangle}$, where $\alpha=(\lambda_0-1)c_0^{-1}$. If $z \in \mathcal{F}_0\cap G_{v}$, then Lemma \ref{new} shows that
    $[z]_{B_0}=\begin{bmatrix}
1 & c \\
0 & \lambda
\end{bmatrix}$,
where $c\in \mathbb{F}_{q}$ and $\lambda\in {\rm det}(G)$. By \eqref{51} and \eqref{titi}, we can conclude $1+\gamma_0c_0=1+\gamma c$ and $(1-\lambda_0)\gamma_0=(1-\lambda)\gamma$, for some  $\gamma_0,\gamma \in \mathbb{F}_{q}^*$.  Then, we have $z(v+(\lambda-1)c^{-1}w)=\lambda(v+(\lambda-1)c^{-1}w)$. However,  $(\lambda_0-1)c_0^{-1}=(\lambda-1)\gamma\gamma_0^{-1}c^{-1}\gamma^{-1}\gamma_0=(\lambda-1)c^{-1}$.  Hence, $z \in G_{\langle w_{\alpha}\rangle}$. Consequently, $\mathcal{F}_0\cap G_{v} \subseteq G_{\langle w_{\alpha} \rangle}$, where $\alpha =(\lambda_0-1)c_0^{-1}$. Assume that  $t \in \mathbb{F}_q^{\ast}$ and $ x_t \in \mathcal{F}_0\cap G_{w_t} $. Set $B_t=\{v,w_t\}$. Then, $B_t$ is a basis of $V$. One can check at once that $[x_0]_{B_t}=\begin{bmatrix}
1 & 1+tc_0-\lambda_0 \\
0 & \lambda_0
\end{bmatrix}$. So, \eqref{51} shows that $
[x_t]_{B_t}=\begin{bmatrix}
1+(1+tc_0-\lambda_0)\gamma_t & 0 \\
(\lambda_0-1)\gamma_t & 1
\end{bmatrix}$, for some $\gamma_t\in \mathbb{F}_q^{\ast}$.
Also, $v+ \alpha w=(1 -\alpha t^{-1})v+\alpha t^{-1} w_t$ and
$x_t ((1-\alpha t^{-1})v+\alpha t^{-1}
w_t)=(1+(1+tc_0-\lambda_0)\gamma_t)((1-\alpha t^{-1})v+\alpha
t^{-1} w_t)$. Hence, $x_t \in G_{\langle w_{\alpha}\rangle}$.
Thus, $\mathcal{F}_0\cap G_{w_t} \subseteq G_{\langle
w_{\alpha}\rangle}$. The same reasoning shows that
$\mathcal{F}_0\cap G_{w} \subseteq G_{\langle w_{\alpha}\rangle}$.
Thus, $\mathcal{F} \subseteq G_{\langle w_{\alpha} \rangle}$. So,
Proposition \ref{extra1} forces  $\mathcal{F} \subseteq
\mathcal{G}_{\langle  w_{\alpha} \rangle}$, as desired.
\end{proof}
{\bf Acknowledgement}
The author is grateful to Professor Karen Meagher for drawing his attention to the flaw in \cite{Ahanjideh}, and  constructive  comments that improved the paper.  Also, the author sincerely thanks the referee for his/her invaluable comments and suggestions.

\end{document}